\documentclass[11pt]{amsart}
\usepackage{amsmath, amsthm, amssymb,xspace}
\usepackage[breaklinks=true]{hyperref}
\usepackage{amsmath,amscd}

\theoremstyle{plain}
\newtheorem{theorem}[equation]{Theorem}
\newtheorem{lemma}[equation]{Lemma}
\newtheorem{prop}[equation]{Proposition}
\newtheorem{cor}[equation]{Corollary}

\theoremstyle{definition}

\newtheorem{remark}[equation]{Remark}

\newcommand{\bggo}{\mathcal O}

\newcommand{\mf}[1]{\displaystyle{\mathfrak{#1}}}

\newcommand{\comment}[1]{}

\DeclareMathOperator{\spec}{\ensuremath{Spec}}

\DeclareMathOperator{\Gr}{\ensuremath{gr}}

\DeclareMathOperator{\Hom}{\ensuremath{Hom}}

\begin{document}

\title{Ideals in deformation quantizations over $\mathbb{Z}/p^n\mathbb{Z}$}
\author{Akaki Tikaradze}
\address{The University of Toledo, Department of Mathematics, Toledo, Ohio, USA}
\email{\tt tikar06@gmail.com}
\maketitle

\begin{abstract}
Let $\bf{k}$ be a perfect field of characteristic $p>2.$ 
Let $A_1$ be an Azumaya algebra over a smooth symplectic affine variety over $\bf{k}.$ Let $A_n$ be a
deformation quantization of $A_1$ over $W_n(\bf{k})$. We prove that all $W_n(\bf{k})$-flat two-sided
ideals of $A_n$ are generated by central elements.

\end{abstract}

\vspace*{0.5in}

Let $\bf{k}$ be a perfect field of characteristic $p>2.$
 For $n\geq 1$, let $W_n(\bold{k})$ denote  the ring of length $n$ Witt vectors over $\bf{k}.$
Also, $W(\bold{k})$ will denote the ring of Witt vectors over $\bf{k}$. As usual, given an algebra $B$ its center
will be denoted by $Z(B).$
 Throughout the paper we will fix once and for all an affine smooth symplectic variety $X$ over $\bf{k},$
 and an Azumaya algebra $A_1$ over $X$ (equivalently over $\bggo_X.)$ 
Thus, we may (and will) identify the center of $A_1$ with $\bggo_X$-the structure ring of $X: Z(A_1)=\bggo_X.$
 Let  $\lbrace, \rbrace$ denote the corresponding 
Poisson bracket on $\bggo_X.$ A deformation quantization of  $A_1$ over $W_n(\bold{k}), n\geq 1$ is, by definition, a flat associatiove $W_n(\bf{k})$-algebra $A$ 
equipped with an isomorphism $A/pA\simeq A_1$ such that for any $a, b\in A$ such that
$a\ mod\ p\in \bggo_X, b\ mod\ p\in \bggo_X,$ one has 
$$\lbrace a\ mod\ p, b\ mod\ p\rbrace=(\frac{1}{p} [a, b])\ mod\ p.$$
One defines similarly a quantization of $A_1$ over $W(\bold{k}).$

Main result of this note is the following

\begin{theorem}\label{azumaya}
Let $A$ be a deformation quantization over $W_n(\bf{k})$ of an Azumaya algebra $A_1$ over $X.$
 Let $I\subset A$ be a two-sided ideal which is flat over $W_n(\bold{k}).$ Then $I$ is generated
 by central elements: $I=(Z(A)\cap I)A.$ 

\end{theorem}
\footnote{We showed in \cite{Ti} that Hochschild cohomology of a quantization $A$ is isomorphic to the de Rham-Witt
 complex $W_n\Omega^{*}_X$ of $X$}
Before proving this result we will need to recall some results of Stewart and Vologodsky \cite{SV} on centers of
certain algebras over $W_n(\bf{k})$. 

Throughout for an associative flat $W_n(\bf{k})$-algebra $R,$ we will denote its reduction
$\mod p^m$ by $R_m=R/p^mR.$ Also center of an algebra $R_m$ will be denoted by $Z_m, m\leq n.$
Recall that in this setting there is the natural deformation Poisson bracket on $Z_1$ defined as follows.
Given $z, w\in Z_1,$ let $\tilde{z},\tilde{w}$ be lifts in $R$ of $z, w$ respectively. Then put 
$$\lbrace z, w\rbrace=\frac{1}{p}[\tilde{z},  \tilde{w}] \mod p.$$
 In this setting, Stewart and Vologodsky [\cite{SV} formula (1.3)] 
constructed a ring homomorphism $\phi_m:W_m(Z_1)\to Z_m$ from the ring of length $m$ Witt vectors over $Z_1$ to $Z_m,$
 defined as follows
$$\phi_n(z_1,\cdots, z_m)=\sum_{i=1}^mp^{i-1}\tilde{z_i}^{p^{m-i}}$$
where $\tilde{z_i}\in R$ is a lift of $z_i, 1\leq i\leq m$. We also have the following natural maps 
$$r:Z_m\to Z_{m-1}, r(x)= x\mod p^{m-1}, v:Z_{m-1}\to Z_m, v(x)=p\tilde{x}$$
where $\tilde{x}$ is a lift of $x$ in $R_m.$ 
On the other hand on the level of Witt vectors of $Z_1,$ we have the Verschibung map
$V:W_m(Z_1)\to W_{m+1}(Z_1)$ and the Frobenius map $F:W_m(Z_1)\to W_{m-1}(Z_1).$ It was checked in \cite{SV} 
that above maps commute 
$$\phi_{m-1}F=r\phi_m,\quad \phi_mV=v\phi_{m-1}.$$

We will recall the following crucial computation from \cite{SV}.
 Let $x=\phi_m(z), z=(z_1,\cdots,z_m)\in W_m(Z_1)$ and let $\tilde{x}$ be a lift of $x$ in $R_{m+1}.$
Then it was verified in \cite{SV} that the following inequality holds in $Der_{\bf{k}}(Z_1, Z_1)$
\begin{equation}\label{1}
\delta_x=(\frac{1}{p^m}[\tilde{x},-])\mod p|_{Z_1}=\sum_{i=1}^m z_i^{p^{m-i}-1}\lbrace z_i,-\rbrace 
\end{equation}

 The main result of [\cite{SV}, Theorem 1] states that if $\spec Z_1$ is smooth variety and 
the deformation Poisson bracket on $Z_1$ is induced from a symplectic form on
$\spec Z_1,$ then the map $\phi_m$ is an isomorphism for all $m\leq n$. In particular,
$${Z_1}^{p^m}=Z_{m+1}\mod p.$$
We will need the following slight generalization of this result. Its proof follows very closely to the
one in [\cite{SV}, Theorem 1].

\begin{prop}\label{center}

 Let $n\geq 1$ and $m\subset \bggo{}_X$ be an ideal, and let $B=\bggo{}_X/m^{p^{n}}\bggo{}_X.$
Let $R$ be an associative flat $W_n(\bold{k})$-algebra such that $Z(R/pR)=B$ and 
the corresponding deformation Poisson bracket
 on $B$ coincides with the one induces from $X.$ Then 
 $$Z(R)=\phi_n(W_n(B)),\quad Z(R)\cap pR=\phi_n(VW_{n-1}(B)).$$

\end{prop}

Just as in [\cite{SV}, Lemma 2.7] the following result plays the crucial role.

\begin{lemma}\label{muh}

Let $z_1,\cdots, z_n\in B$ be such that $\sum_{i=1}^{n} z_i^{p^{n-i}-1}dz_i=0.$
Then $z_i\in B^p+{\bar{m}}^{p^{i}}B,$ where $\bar{m}=m/m^{p^n}\bggo{}_X.$
\end{lemma}

\begin{proof}
Put $S=\bggo_X.$
We will proceed by induction on $n.$ Let $n=1.$ Let $z_1$ be a lift of $x_1$ in $S.$
Thus $dz_1\in m^p\Omega^1_S\cap dS.$ Since $\Omega^1_S/dS$ is a flat $S^p$-module, it follows that
 $$m^p\Omega^1_S\cap dS=m^pdS=d(m^pS).$$ So, $dx_1\in d(m^pS)$. Hence $$x_1\in m^pS+Ker(d)=m^pS+S^p.$$
Assume that our statement is true for $n-1.$
Let $x_i\in S$ be a lift of $z_i, 1\leq i\leq n.$ Thus $$\sum_{i=1}^{n} x_i^{p^{n-i}-1}dx_i\in m^{p^{n}}\Omega^1_S.$$
As usual $Z^1(\Omega_S)$ will denote $Ker(d)\subset \Omega^1_S.$
Remark that $$Z^1(\Omega_S)\cap m^{p^{n}}\Omega^1_S=m^{p^{n}}\Omega^1_S,$$ this follows from flatness of
 $\Omega^1_S/Z^1(\Omega_S)$ over $S^{p^{n}}.$ Thus  $$\sum x_i^{p^{n-i}-1}dx_i\in m^{p^{n}}Z^1(\Omega^{*}_S).$$
Recall that the inverse Cartier map $C^{-1}:\Omega^1_S\to \Omega^1_S/dS$ is defined as follows 
$$C^{-1}(fdg)=f^pg^{p-1}dg.$$
Recall also that smoothness of $S$ implies that $C^{-1}$ defines an isomorphism onto
 $Z(\Omega^1_S)/dS,$  Thus we may write 
$$C^{-1}(\sum_{i<n}x_i^{p^{n-1-i}}dx_i)=C^{-1}(x)+dx', x\in m^{p^{n-1}}\Omega^1_S, x'\in S.$$
Injectivity of $C^{-1}$ implies that $$\sum_{i<n}x_i^{p^{n-i-1}}dx_i\in m^{p^{n-1}}\Omega^1_S.$$ Thus by 
induction assumption, we get that  $x_i\in S^p+m^iS, i<n.$ Therefore $dx_n\in m^{p^{n}}S$, so
$x_n\in S^p+m^{p^{n}}S$
\end{proof}

\begin{lemma}\label{key}
Let $z_1,\cdots, z_n\in B$ be such that $\sum_{i=1}^nz_i^{p^{n-i}-1}\lbrace z_i, -\rbrace=0.$ Then 
$z_i\in B^p+m^{p^{i}}B.$
\end{lemma}

\begin{proof}

Symplectic form $\omega\in \Omega^2_S$ gives an isomorphism $$\iota: \Omega^1_S\to Der{\bf{k}}(S,S)=Hom_S(\Omega^1_S, S)$$
such that $\iota(gdf)=g\lbrace f, -\rbrace, f, g\in S.$ Since $\Omega^1_B=\Omega^1_S\otimes_{S^{p^{n+1}}}B$, we get an isomorphism
$\bar{\iota}:\Omega^1_B\to Der_{\bf{k}}(B, B)$ defined as $\bar{\iota}(xdy)=x\lbrace y, -\rbrace.$ Thus applying $\bar{\iota}^{-1}$ to the equality
$$\sum_{i=1}^nz_i^{p^{n-i}-1}\lbrace z_i, -\rbrace=0,$$ we obtain that $$\sum_{i=1}^n z_i^{p^{n-i}-1}dz_i=0.$$ Hence by Lemma \ref{muh}
 we are done.
\end{proof}

\noindent Once Lemma \ref{key} is established, the proof of Proposition \ref{center} is identical to the one in \cite{SV}.
Indeed, by induction assumption, $\phi_{n-1}:W_{n-1}(B)\to Z(R_{n-1})$ is surjective. 
Let $x\in Z(R)$. Then there exists $z=(z_1,\cdots,z_{n-1})\in W_{n-1}(B)$ such that
 $$x'=x\mod p^{n-1}=\phi_{n-1}(z).$$ 
Hence by equality \eqref{1}, we have
$$0=\delta_{x'}=\sum_{i=1}^{n-1}z_i^{p^{n-1-i}}\lbrace z_i,-\rbrace=0.$$
Therefore by Proposition \ref{key}, we get that $z_i=a_i^p+b_i$, where $b_i\in \bar{m}^{p^{i}}B.$
Thus, $$\phi_{n-1}(z)=\phi_{n-1}(a_1^p,\cdots,a_{n-1}^p)\in \phi_n(W_{n}(B))\mod p^{n-1}.$$
Hence $$Z(A)\subset \phi_{n}(W_n(B))+p^{n-1}R.$$ Now since $v:Z(R_{n-1})\to Z(A)\cap pR$
is an isomorphism, surjectivity of $\phi_{n-1}$ implies that $\phi_n(W_n(B))=Z(R)$ and 
$\phi_n(VW_{n-1}(B))=Z(R)\cap pR.$ This concludes the proof of Proposition \ref{center}.







Now we can prove our main result.

\begin{proof}[Proof of Theorem \ref{azumaya}]

We proceed by induction on $n.$ When $n=1,$ the statement is a standard property
of Azumaya algebras [\cite{MR}, Proposition 7.9]. We will assume that the statement holds for $n$ and prove it for $n+1.$
Thus, $A$ is a quantization of $A_1$ over $W_{n+1}(\bf{k})$, and $I\subset A$ is a $W_{n+1}(\bf{k})$-flat
two-sided ideal.
Let us identify $Z_m=Z(A/p^mA)$ with $W_m(Z_1)$ via the isomorphism $\phi_m,m\leq n+1.$
We will put $I_i=I\mod p^i.$ Since by the assumption $I_i, A_i$ are free $W_i(\bold{k})$-modules,
it follows that $A_i/I_i$ is a free $W_i(\bold{k})$-module 
for all $i\leq n+1.$
 Put $m_{n}=Z_{n}\cap I_{n}.$ 
Recall that $r^i:A_m\to A_{m-i}, i\leq m$ denotes the projection map, while $v^i:A_{m-i}\to A_m$
denotes the multiplication by $p^i.$
For $i\leq n$, let us put $m_i=m_{n}\mod p^i.$ So $m_i$ is an ideal in $Z_i.$
It follows from the inductive 
assumption that for all $i\leq n,$ we have 
$I_i=m_iA_i.$ Moreover, since
$A_1$ is an Azumaya algebra, we have $m_1=m_1A_1\cap Z_1.$

 Since $A/I$ is a free  $W_{n+1}(\bold{k})$-module, we have a short exact sequence
$$
\begin{CD}
0\to I_1 @>{{v}}^{n}>> I@>r >> I_{n} \to 0 .
\end{CD}
$$
We claim that for all $x\in m_{n}$, we have  $${F}^{n-1}d(x)\subset m_1\Omega^1_{Z_1},$$
here 
$$F:W_n\Omega^{*}_{Z_1}\to W_{n-1}\Omega^{*}_{Z_1}$$
 is the Frobenius map of the de Rham-Witt complex
of $Z_1.$ Indeed, it follows from the above exact sequence that if $\tilde{x}$ is a lift of $x$ in $A_{n+1}$,
then $\delta_x=\frac{1}{p^n}ad(\tilde{x}):A_1\to A_1$ is a derivation such that $Im(\delta_x)\subset I_1$.
In particular, $\delta_x|_{Z_1}$ is a derivation of $Z_1$ whose range lies in the ideal $m_1.$
Therefore, by equality \eqref{1} it follows that the derivation $\delta_x$ corresponds to $F^n(dx)\in \Omega^1_{Z_1}$
under the identification $Der(Z_1)=\Omega^1_{Z_1}$ by the symplectic form on $X.$
Thus we obtain the desired inclusion. Recalling that $m_1=F^{n-1}(m_n)Z_1$, we get 

 $$F^{n-1}d(m_{n})\subset F^{m-1}(m_{n})\Omega^1_{Z_1}.$$ 




Next we will use the following

\begin{lemma}\label{frob}

Let $S$ be a smooth essentially of finite type commutative ring over $\bold{k}.$ Let $m$
be an ideal in $W_{n}(S)$ such that $F^{n-1}(dm)\subset F^{n-1}(m)\Omega^1_{S}.$ Put $\bar{m}=m \mod VW_{n-1}(S).$
Then $\bar{m}=(\bar{m}\cap S^p)S.$

\end{lemma}

\begin{proof}

At first we will show that $$F^{n-1}(m)\Omega^1_S\cap F^{n-1}(W_{n}\Omega^1_S)=F^{n-1}(mW_{n}\Omega^1_S).$$
For simplicity we put $N_n=F^{n-1}(W_{n}\Omega^1_S).$ Thus, we want to show that 
$$\bar{m}^{p^{n-1}}\Omega^1_S\cap N_n=\bar{m}^{p^{n-1}}N_n.$$ It follows from our assumptions on $S$ that $\Omega^1_S/N_n$ is a flat
$S^{p^{n-1}}$-module [\cite{Il}, Proposition 2.2.8  and isomorphism (3.11.3)].
Therefore $N_n/\bar{m}^{p^{n-1}}N_n$ injects into $\Omega^1_S/\bar{m}^{p^{n-1}}\Omega^1_S.$ Hence
$$\bar{m}^{p^{n-1}}\Omega^1_S\cap N_n=\bar{m}^{p^{n-1}}N_n.$$
Thus $$F^{n-1}(dm)\subset F^{n-1}(mW_{n}\Omega^1_{S}).$$ Recall that $Ker F^{n-1}=V(W_{n-1}\Omega^1_S).$ Hence we can conclude that
$$dm\subset mW_{n}\Omega^1_S+V(W_{n-1}\Omega^1_S).$$ Therefore, $d\bar{m}\subset \bar{m}\Omega^1_S.$





Now we claim that $\bar{m}=(\bar{m}\cap S^p)S.$ Since the statement is local, we may assume that $S$ is a regular local ring.
It follows that any derivation $g:S\to S$ preserves $m: g(\bar{m})\subset \bar{m}.$ 
Thus, $I$ is a submodule of $S$ viewed as a $\Hom_{S^p}(S, S)$-module.
Since $S$ is a finite rank free $S^p$-module, $\Hom_{S^p}(S, S)$ is
a matrix algebra over $S^p$ and the claim follows.

So $\bar{m}={m'}^pS$ for some ideal
$m'\subset S.$

 \end{proof}
Thus we can conclude using Lemma \ref{frob} that $m_n \mod VW_{n-1}(Z_1)$ is generated by elements in $Z_1^p.$
Since $m_1=F^{n-1}(m_n)Z_1,$ we get that  $m_1=(m_1\cap Z_1^{p^n})Z_1.$ Thus $m_1=l^{p^n}Z_1$
for some ideal $l\subset Z_1.$
We have a short exact sequenced 

$$
\begin{CD}
0\to I_n @>{{v}}>> I@>\mod p >> I_{1} \to 0 .
\end{CD}
$$
Recall that $I_n=m_nA_n, m_n\subset Z_n.$ 
 Let $x\in m_1.$ There 
there exist $z\in Z(A)$ and $y\in A$ such that $z+py\in I$ and $z\mod p=x.$
Therefore $$[py, A]\subset I.$$ Hence $$py\in Z(A/I)\cap pA/I.$$ 
Applying Proposition \ref{center} to $R=A/I, m=l^{p^n},$
we may write
$$py=\sum_{i\geq 1}p^i{z}_i^{p^{n-i}}+py'$$ where $z_i \mod p\in Z_1, y'\in I.$ 
Thus $$z+\sum_{i\geq 1}p^iz_i^{p^{n-i}}+py'\in I,$$ moreover $y'\in I$ and $z\mod p=x.$
Thus $$z'=z+\sum_ip^iz_i^{p^{n-i}}\in Z(A)\cap I, x=z' \mod p.$$ Hence $$I_1=(I\cap Z(A))A \mod p.$$
 On the other
hand $$pI=vI_n=v((I_n\cap Z_n)A_n)\subset (I\cap Z)A.$$ Therefore $I=(Z(A)\cap I)A.$

\end{proof}

Typical setting where Theorem \ref{azumaya} can be used is as follows. Let $Y$ be a smooth affine
variety over $W_n(\bf{k}).$ Let $\bar{Y}$ denote the $\mod p$ reduction of $Y.$ Let $D_{Y}$ (respectively $D_{\bar{Y}}$)
denote the ring of crystalline (PD) differential operators on $Y$ (respectively $\bar{Y}$). Then $D_{\bar{Y}}=D_Y/pD_Y$
is an Azumaya algebra over $X=T^{*}(\bar{Y})^{(1)}$-the Frobenius twist of the cotangent bundle of $\bar{Y}$ (by \cite{BMR}).
Also, it follows from \cite{BK} that the corresponding deformation Poisson bracket on $X$ is induced from the symplectic
form on $T^{*}(\bar{Y}).$ Thus, Theorem \ref{azumaya} applies to $A=D_{Y}.$

\begin{remark}

In Theorem \ref{azumaya} it is necessary to assume that ideal $I$ be flat over $W_n(\bf{k}).$ 
Indeed, let $m$ be an ideal in $Z_1$, such that $(m\cap Z_1^{p^{n-1}})Z_1\ne m$. Let $I$ be a preimage of $mA_1$ in $A$ under $\mod p$ reduction map. Then since $$Z(A) \mod p=Z_1^{p^{n-1}},$$ 
it follows that $I\neq (I\cap Z(A))A.$
In particular, $A$ is not an Azumaya algebra over $Z(A)$ for $n\geq 2.$

\end{remark}

In what follows we will assume that the ground field $\bold{k}$ is algebraically closed.
As a consequence of Theorem \ref{azumaya}, we have the following criterion for (topological) simplicity
of $W(\bold{k})$-algebras.

\begin{cor}\label{simple}

Let $A$ be a quantization of $A_1$ over $W(\bold{k})$
Then $Z(A)=W(\bf{k})$ and algebra $A[p^{-1}]$ is (topologically) simple.

\end{cor}

\begin{proof}

As before, we will put $A_n=A/p^nA, n\geq 1.$ 
Let us denote by $r_i$ the quotient map $A\to A/p^iA.$ As before, $r^i:A_{n+i}\to A_{n}$ denotes
the quotient map.
Then it follows that $r^n(Z_{n+1})= Z_1^{p^n}.$ Hence 
$$r_1(Z(A))\subset \cap_{i=1}^{\infty} Z_1^{p^i}=\bold{k}.$$
Hence $Z(A)\subset W(\bold{k})+pA$, which implies that $$Z(A)\subset \cap_{n=1}^{\infty}(W(\bold{k})+p^nA)=W(\bold{k}).$$
Let $I\subset A[p^{-1}]$ be a closed two-sided ideal. Put $I'=I\cap A[p^{-1}].$ Then $I'$ is a topologically free $W(\bold{k})$-module.
Thus, $I'_n=r_n(I')$ is a two-sided ideal of $A_n$ which is free as a $W_n(\bf{k})$-module. 
Let us put $m_n=I'_n\cap  Z_n, n\geq 1.$ Using Theorem \ref{azumaya} we obtain that for all $n\geq 1$
$$m_1={r}^n(m_{n+1})Z_1.$$ 
So, $$m_1=(m_1\cap Z_1^{p^n})Z_1.$$ Let us put $$m_1\cap Z_1^{p^n}=l_n^{p^n},\quad l_n\subset Z_1.$$ Clearly 
ideals $l_n$ form an ascending chain: $$l_n\subset l_{n+1}\subset \cdots.$$ 
Thus $\cup_{n=1}^{\infty}l_n=l_i,$ for some $i.$ To summarize, $m_1=l_i^{p^n}Z_1$ for all $n\geq i.$ Therefore
$m_1^p=m_1.$
 Thus, either $m_1=0$ or $m_1=Z_1.$ Therefore, $I'=0$ or $I'=A.$
\end{proof}

The next result provides a criterion for simplicity of algebras defined over global rings.
Let $R$ be a commutative domain. We will say that an $R$-algebra $A$ has a generic freeness property over $R$ if
for any finitely generated left $A$-module $M$, there exists a nonzero element $f\in R$ such that $M_f$ is a free
$R_f$-modules.
\begin{cor}\label{Simple}

Let $R $ be a finitely generated subring of $\mathbb{C},$ and $F$ its field of fractions. 
Let $S$ be an $R$-algebra
 which has generic freeness property over $R.$ Assume that for all nonzero $f\in R$ and 
 for infinitely many primes $p,$
 there exists an algebraically closed field $\bold{k}$ of characteristic $p$ and a homomorphism
 $\rho:R_f\to W(\bold{k}$), such that 
 $S\otimes_R\bold{k}$ is an Azumaya algebra and $\spec Z(S\otimes_R\bold{k})$ equipped with 
 the deformation Poisson bracket  is a smooth
 symplectic variety over $\bold{k}.$ Then algebra $S\otimes_RF$ is a simple.

\end{cor}

\begin{proof}
Assume that algebra $S\otimes_RF$ is not simple.
Then there exists  an $R$-torsion free nonzero proper ideal $I\subset S$  such that $I\cap R=0$. 
By localizing $R$ we may assume by generic flatness of $S$ that $S/I$ is a nonzero free $R$-module. Let $p$ be
a prime and let $\rho:R\to W(\bold{k})$ be a homomorphism as in the statement.
 Denote by $A$ the $p$-adic completion of $S\otimes_RW(\bold{k}),$
and denote the $p$-adic completion of $I\otimes_RW(\bold{k})$ by $\bar{I}.$
 Thus, algebra $A$ satisfies assumptions of  Corollary \ref{simple}. Hence $\bar{I}[p^{-1}]=A[p^{-1}].$
In particular, $A/\bar{I}=(S/I)\otimes_RW(\bold{k})$ is a nonzero $p$-torsion $W(\bf{k})$-module, a contradiction.
\end{proof}

\end{document}